\documentclass[leqno,centertags,cmex10]{amsart}

\usepackage{amsfonts,amssymb,curves,latexsym,hyperref}
\usepackage[dvipsnames,svgnames,x11names,hyperref]{xcolor}
\usepackage{blindtext}
\usepackage{graphicx}
\usepackage{amssymb}
\usepackage{bm}

\usepackage[ddmmyyyy]{datetime}
\hypersetup{colorlinks,breaklinks,
             urlcolor=blue,
             linkcolor=blue}
\newtheorem{theorem}{Theorem}
\newtheorem{proposition}[theorem]{Proposition}
\newtheorem{lemma}[theorem]{Lemma}

\newtheorem{remark}[theorem]{Remark}
\newtheorem{corollary}[theorem]{Corollary}
\newtheorem{definition}[theorem]{Definition}

\def\Fq{{\mathbb F}_q}

\def\F{\mathsf{F}}

\newcommand{\K}{\mathbb{K}}
\newcommand{\I}{\mathcal{I}}

\newcommand{\N}{\mathbb{N}}

\renewcommand{\-}{\backslash }
\newcommand{\Z}{\mathbb{Z}}

\begin{document}

\title[Higher weight spectra of Veronese codes]{Higher weight spectra of Veronese codes}

\author[Johnsen]{Trygve Johnsen}
\address{Department of Mathematics and Statistics, 
 UiT-The Arctic University of Norway  \newline \indent 
N-9037 Troms{\o}, Norway}
\email{trygve.johnsen@uit.no}

\author[Verdure]{Hugues Verdure}
\address{Department of Mathematics and Statistics, 
 UiT-The Arctic University of Norway  \newline \indent 
N-9037 Troms{\o}, Norway}
\email{hugues.verdure@uit.no}
\thanks{T. Johnsen and H. Verdure are with the Department
of Mathematics and Statistics, UiT The Arctic University of Norway, 9037 Troms\o, Norway. The authors are grateful to Sudhir Ghorpade, IIT-Bombay for being a valuable discussion partner as the work proceeded,  and for the great hospitality, which provided optimal conditions for making the this article possible. \newline This work was partially supported with a joint grant from RCN, Norway (Project number 280731), and DST, India.\newline This work was also partially supported by the project Pure Mathematics in Norway, funded by Bergen Research Foundation and Troms{\o} Research Foundation}

\subjclass[2000]{05E45, 94B05, 05B35, 13F55}
\date{\today}

\begin{abstract}
We study $q$-ary linear codes $C$ obtained from  Veronese surfaces over finite fields.
We show how one can find the higher weight spectra of these codes, or equivalently, the
weight distribution of all extension codes of $C$ over all field extensions of $\Fq$. 
Our methods will be a study of the Stanley-Reisner rings of a series of matroids associated to each code $C$.
\end{abstract}

\maketitle

\section{Introduction}

Projective Reed-M\"uller codes is a class of error-correcting codes that has attracted much attention over the last decades.
To find the code parameters, including the generalized Hamming weights, has been a difficult task, and important results concerning this, have appeared quite recently. To find the higher weight spectra of such codes is more difficult, when the order of the Reed-M\"uller codes is higher than one, and to our knowledge there are few results about  this. Therefore it is natural to start with the simplest projective Reed-M\"uller codes of order at least $2$, namely the so-called Veronese codes $C_q$ over any finite field $\Fq$, where the $n=q^2+q+1$ columns of the generator matrix $G_q$ correspond to the points of $\mathbb{P}^2$. Moreover each row  is obtained by  taking an element of a basis for the vector space of all homogeneous polynomials of degree $2$ in $3$ variables, and evaluating it at the points of $\mathbb{P}^2$ (in some fixed order). Since this vector space has dimension $6$, there will be $6$ such rows.
Alternatively one could think of the columns of $G_q$ as the point of the $2$-uple Veronese embedding of $\mathbb{P}^2$ in
$\mathbb{P}^5.$  This is why we call these codes Veronese codes; since they in the way described correspond to the projective system of points of the mentioned Veronese surface (of degree $4$ in $\mathbb{P}^5$).

In this article, we are interested in computing the higher weight spectra, that is the number of subcodes of given dimension and weight of $C_q$.

The code $C_2$ is MDS and of dimension $6$ and length $7$, while the code $C_3$ of dimension  $6$ and length $13$ is more interesting, and it differs both from $C_2$, and from the codes $C_q$, for $q \ge 4$, concerning the aspects we study here. We determine the higher weight spectra and  the generalized weight polynomials for both codes.  For the codes $C_q$, with $q \geqslant 4$, we give a unified treatment, and determine both their higher weight spectra and their generalized weight polynomials.  All elements of the weight spectra,
and all coefficients of the generalized weight polynomials, turn out to be polynomials in $q$, with coefficients $\frac{a}{b}$,
where $a$ is an integer, and $b$ an integer dividing $24$. 

Our methods will consist of finding the $\mathbb{N}$-graded resolutions of the Stanley-Reisner rings of a series of matroids derived from the parity check matroid $M_q$ of each code. The $\mathbb{N}$-graded Betti numbers of these resolutions will give us the generalized weight polynomials $P_j(Z)$ that calculate the usual weight distribution of all extension codes of the $C_q$ over field extensions of $\Fq$.
Finally a straightforward and well known conversion formula will, from the knowledge of the $P_j(Z)$,  give us the higher weight spectra of the original codes $C_q$ that we study.

\section{Definitions and notation}
Let $q$ be a prime power and let $\nu_q$ be the Veronese  map that maps $\mathbb{P}^2$ into $\mathbb{P}^5$ over $\Fq$, i.e.
$(x,y,z)$ is mapped to  $(x^2,xy,xz,y^2,yz,z^2),$ and let $V_q$ be the image,  a non-degenerate smooth surface of degree $4$. The cardinality $|V|$ of $V$ is  $|\mathbb{P}^2|= q^2+q+1$. Fix some order for the points of 
$V$, and for each such point, fix a coordinate $6$-tuple that represents it.
Let $G_q$ be the $(6 \times (q^2+q+1))-$ matrix, whose columns are the coordinate $6$-tuples  of the points of $V$, taken
in the order fixed.
\begin{definition}\label{def:Veronese}
The Veronese code $C_q$ is the linear $[q^2+q+1,6]_q$-code with generator matrix $G_q$.
\end{definition}

For $q=2$ we thus get a $[7,6]$-code $C_2$, and it is well known, for example by looking at its dual code,
which is generated by a single code word with no zeroes (\cite{So}), that this is an MDS-code, and then 
all we are interested to know about this code is well known (A more straightforward method is of course just to calculate all $64$ codewords, and check that there is no such word with weight $1$).   From now on we will assume that $q \ge 4$, and we will give a common description of the $C_q$ for all these $q$.
We will return to the cases
$q=2$ and $3$ first in Section \ref{2and3}, where we will comment on,  and give the relevant results for these two cases. 

\subsection{Hamming weights, spectra and generalized weight polynomials}

\begin{definition}\label{def:support}
Let $C$ be a $[n,k]$ linear code over $\Fq$. Let $\bm{c} = (c_1,\cdots,c_n) \in C$. The Support of $\bm{c}$ is the set \[Supp(\bm{c}) = \{ i \in \{1,\cdots,n\}: c_i \neq 0\}.\] Its weight is \[wt(\bm{c}) = |Supp(\bm{c})|.\] Similarly, if $T \subset C$, then its support and weight are \[Supp(T) = \bigcup_{\bm{c} \in T} Supp(\bm{c}) \textrm{ and } wt(T) = |Supp(T)|.\]
\end{definition}
Important invariants of a code are the generalized Hamming weights, introduced by Wei in~\cite{W}: 
\begin{definition}\label{def:hamming}
Let $C$ be a $[n,k]$ linear code over $\Fq$. Its generalized Hamming weights are \[d_i = \min\{wt(D): D \subset C\textrm{ is a subcode of dimension }i\}\] for $1 \leqslant i \leqslant k$.
\end{definition}
We also have 
\begin{definition}\label{def:spectra}
Let $C$ be a $[n,k]$ linear code over $\Fq$. For $1\leqslant w \leqslant n$ and $1 \leqslant d \leqslant k$, the higher weight spectra of $C$ are \[A_w^{(r)} = \left|\{D: D \textrm{ subcode of }C\textrm{ of dimension }r \textrm{ and weight }w\}\right|.\]
\end{definition}
In particular, we have \[d_r= \min\{w:\ A_w^{(r)} \neq 0\}.\]
In~\cite{JP}, Jurrius and Pellikaan show that the number of codewords of a given code extended to a field extension of a given weight can be expressed by polynomials (the generalized weight polynomials). More precisely, if $C$ is a $[n,k]$-code over $\Fq$, then the code $C^{(i)}=C \otimes_{\Fq} \F{q^i}$ for $i \geqslant 1$ is a $[n,k]$ code over $\F{q^i}$. Any generator/parity check matrix of $C$ is a generator/parity check matrix of $C^{(i)}$. Then \begin{theorem}\label{th:Jurrius} Let $C$ be a $(n,k)$-code over $\Fq$. Then, there exists polynomials $P_w \in \Z[Z]$ for $0\leqslant w \leqslant n$ such that \[\forall i \geqslant 1, P_w(q^i) = \left|\left\{ \bm{c} \in C^{(i)}: wt(\bm{c}) = w\right\}\right|.\] 
\end{theorem} In~\cite{J}, Jurrius gives a relation between the higher weight spectra and the polynomials defined above, namely
\begin{theorem}\label{th:relspecpol}Let $C$ be a $[n,k]$ code over $\Fq$. Let $0\leqslant w \leqslant n$. Then \[P_w(q^m) = \sum_{r=0}^m A_w^{(r)} \prod_{i=0}^{r-1}(q^m-q^i).\]
\end{theorem}

\subsection{Matroids, resolutions and elongations}

Our goal in this paper is to find the higher weight spectra for the Veronese codes $C_q$ for $q \geqslant 3$. In order to do this, we will compute the higher weight polynomials of the code, making use of some machinery related to matroids associated to the code and their Stanley-Reisner resolutions.

There are many equivalent definitions of a matroid. We refer to~\cite{Oxley} for a deeper study of the theory of matroids. 

\begin{definition}\label{def:matroid}
A matroid is a pair $(E,\I)$ where $E$ is a finite set and $\I$ is a set of subsets of $E$ satisfying \begin{itemize}
\item[($R_1$)] $\emptyset \in \I$
\item[($R_2$)] If $I \in \I$ and $J \subset I$, then $J \in \I$
\item[($R_3$)] If $I,J \in \I$ and $|I| < |J|$, then $\exists j \in J \- I$ such that $I \cup \{j\} \in \I.$
\end{itemize}

The elements of $\I$ are called independent sets. The subsets of $E$ that are not independent are called dependent sets, and inclusion minimal dependent sets are called circuits.

For any $X \subset E$, its rank is \[r(X) = \max\{|I|: I \in \I,\ I \subset X\}\] and its nullity is $n(X)=|X|-r(X)$. The rank of the matroid is $r(M)=r(E)$. Finally, for any $0 \leqslant i \leqslant |E|-r(M)$, \[N_i = n^{(-1)}(i).\]
\end{definition}

If $C$ is a $[n,k]$-linear code given by a $(n-k) \times k$ parity check matrix $H$, then we can associate to it a matroid $M_C=(E,\I)$, where $E=\{1,\cdots ,n\}$ and $X \in \I$ if and only if the columns of $H$ indexed 
by $X$ are linearly independent over $\Fq$. It can be shown that this matroid is independent of the choice of the parity check matrix of the code. In the sequel, we denote by $M_q$ the matroid associated to the Veronese code $C_q$.

By axioms $(R_1)$ and $(R_2)$, any matroid $M=(E,\I)$ is also a simplicial complex on $E$. Let $\K$ be a field. We can associate to $M$  a monomial ideal $I_M$ in $R=K[\{X_e\}_{e \in E}]$ defined by \[I_M =< \bm{X}^\sigma: \sigma \not \in \I>\] where $\bm{X}^\sigma$ is the monomial product of all $X_e$ for $e \in \sigma$. This ideal is called the Stanley-Reisner ideal of $M$ and the quotient $R_M=R/I_M$ the Stanley-Reisner ring associated to $M$. We refer to~\cite{HH} for the study of such objects. As described in \cite{JV} the Stanley-Reisner ring has minimal $\N$ and $\N^n$-graded free resolutions \[0 \leftarrow R_M \leftarrow R \leftarrow \bigoplus_{j \in \N}R(-j)^{\beta_{1,j}} \leftarrow \bigoplus_{j \in \N}R(-j)^{\beta_{2,j}} \leftarrow \cdots \leftarrow \bigoplus_{j \in \N}R(-j)^{\beta_{|E|-r(M),j}} \leftarrow 0\] and \[0 \leftarrow R_M \leftarrow R \leftarrow \bigoplus_{\alpha \in \N^n}R(-\alpha)^{\beta_{1,\alpha}} \leftarrow \bigoplus_{\alpha \in \N^n}R(-\alpha)^{\beta_{2,\alpha}} \leftarrow \cdots \leftarrow \bigoplus_{\alpha \in \N^n}R(-\alpha)^{\beta_{|E|-r(M),\alpha}} \leftarrow 0.\] 

 In particular the numbers $\beta_{i,j}$ and $\beta_{i,\alpha}$ are independent of the minimal free resolution,
 (and for a matroid also of the field $\K$) and are called respectively the $\N$-graded and $\N^n$-graded Betti numbers of the matroid.
 We have \[\beta_{i,j} = \sum_{wt(\alpha)=j}\beta_{i,\alpha}.\] We also note that $\beta_{0,0}=1$.
 
It is well known that the independent sets of a matroid constitute a shellable simplicial complex. Hence  the ring $R_M$ is Cohen-Macaulay,
and the length $\min \{i | \beta_{i,j} \ne 0,$ for some $j\}$ is $n-r(M)$ by the Auslander-Buchsbaum formula (\cite{AB}).
When $M=M_C$ is associated to the parity check matroid of a linear code of dimension $k$, this length is then $n-(n-k)=k$. 

Moreover, we have, as a direct consequence of  
a more general result (One assumes that $I \subset S$ is a graded ideal such that $R = S/I$ is Cohen–
Macaulay and let $k=projdim(R)$) by Peskine and Szpiro, given in \cite{PS}:

\begin{theorem}\label{th:HK} Let $M$ be a matroid of rank $r=n-k$ on a set of cardinality $n$. Then the $\N$-graded Betti numbers of $R_M$ satisfy the equations
\begin{equation}  \label{HKBS}
\sum_{i=0}^k\sum_{j=0}^n(-1)^ij^s\beta_{i,j}=0,
\end{equation}
for $0 \leqslant s \leqslant k-1,$ where by convention, $0^0=1$.
\end{theorem}
See also \cite[Equation (2.1)]{BS}  and \cite{HeK}.
The $k$ equations (\ref{HKBS}) from Theorem \ref{th:HK} are frequently called the Herzog-K\"uhl equations.
\begin{remark} \label{vandermonde}
For a matroid $M$ we define $\phi_j(M)=\sum_{i=0}^k(-1)^i\beta_{i,j}. $
Then the Herzog-K\"uhl equations can be written:
$$\sum_{j=0}^nj^s\phi_j(M)=0,$$
and it is clear that these equations are independent in the variables $\phi_j(M)$ with a Vandermonde coefficient matrix.
\end{remark}

Also, as explained in~\cite[Theorem 1]{JV}, we can compute the $\N^n$-graded betti number $\beta_{i,\alpha}$ as the Euler characteristic of a certain matroid. If $M$ is a matroid and $\sigma$ is a subset of the ground set $E$, then $M_\sigma$ is the matroid with independent sets \[\I(M_\sigma)= \left\{ \tau \in \I(M): \tau \subset \sigma\right\}.\] Moreover, the Euler characteristic of $M$ is \begin{eqnarray*}\chi(M) &=& \sum_{i=0}^{|E|} (-1)^{i-1} \left|\left\{\tau \subset E: |\tau| = i \textrm{ and } \tau \not\in \I\right\}\right| \\&=& \sum_{i=0}^{|E|} (-1)^i \left|\left\{\tau \subset E: |\tau| = i \textrm{ and } \tau \in \I\right\}\right|\end{eqnarray*} \begin{theorem}\label{th:Euler}
Let $M$ be a matroid on the ground set $E$. Let $\sigma \subset E$. Then \[\beta_{n(\sigma),\sigma}= (-1)^{r(\sigma)-1}\chi(M_\sigma).\] In particular, for any circuit $\sigma$, $\beta_{1,\sigma}=1$.
\end{theorem}
In~\cite{JV} generally for matroids, and in particular for matroids associated to codes, we show that: 
\begin{theorem}\label{thm:AAECC}
Let $C$ be a $[n,k]$-code over $\Fq$. The $\N$-graded Betti numbers of the matroid $M_C$ satisfy:  $\beta_{i,j} \neq 0$ if and only if there exists an inclusion minimal set in $N_i$ of cardinality $j$. In particular, $d_i = \min\{j: \beta_{i,j} \neq 0\}$.
\end{theorem}

\begin{definition}\label{def:elongation}
Let $M=(E,\I)$ be a matroid, with $|E|=n$, and let $l \geqslant 0$. Then,  the $l$-th elongation  of $M$ is the matroid $M^{(l)} = (E,\I^{(l)})$ with \[\I^{(l)} = \{ I \cup X: I \in \I, X \subset E,\ |X| \leqslant l\}.\]
\end{definition}
The $l$-th elongation of $M$ is a matroid of rank $\min\{n,r(M)+l\}$.

\begin{remark}
Another, equivalent, way of defining $M^{(l)}$, is: $M^{(l)}$ is the matroid with the same ground set $E$ as $M$, and with nullity function $n^{(l)}(X)=\max \{0,n(X)-l\},$ for each $X \subset E.$ 
\end{remark}
\begin{definition}
Let $N^{(l)}_i$ be the set of subsets $X$ of $E$ with
$n^{(l)}(X)=i.$
\end{definition}
The following result is trivial, but useful:
\begin{proposition} \label{use}
$N^{(l)}_i=N_{i+l},$ for
$i=0,\cdots,n-r(M)-l.$
In particular the inclusion minimal elements of $N^{(l)}_i$ are the same as the inclusion minimal elements of $N_{i+l}.$
\end{proposition}

The main theorem of~\cite{JRV} gives an expression of the generalized weight polynomials of a code to the Betti numbers of its associated matroid and its elongations, namely:

\begin{theorem}\label{th:Jan} Let $C$ be a $[n,k]$ code over $\Fq$. We denote by $\beta_{i,j}^{(l)}$ the Betti numbers og the matroids $M_C^{(l)}$. Then, for every $ 0 \leqslant w \leqslant n$, 
\[P_w(Z) = \sum_{0 \leqslant l \leqslant k-1}\sum_{i\geqslant 0} (-1)^{i+1}\beta_{i,w}^{(l)}Z^l(Z-1).\]
\end{theorem}
\begin{remark} \label{reform}
{\rm The formula in Theorem~\ref{th:Jan} can also be written
\[P_w(Z)=\sum_{l\geqslant 0}\sum_{i\geqslant 0}(-1)^{i+1}(\beta_{i,w}^{(l-1)}-\beta_{i,w}^{(l)})Z^l.\]
Using Remark \ref{vandermonde} we see that this can be written:
\[P_w(Z)=\sum_{l\geqslant 0}(\phi_w(M^{(l)})-\phi_w(M^{(l-1)}))Z^l.\]
In any case the input in the formula of Theorem \ref{th:Jan} 
contains the output of the Herzog-K\"uhl equations for the various $M^{(l)}$ (when those equations are combined with sufficient other information to be solvable). Whether we want to use the set of all $\beta^l_{i,w}$ as this output/input, or are happy to use just the $\phi_w(M^{(l)})$, is a matter of taste or opportunity.
It is clear that if one knows all the $\beta^l_{i.w}$ for a fixed $w$, then one can derive all the $\phi_w(M^{(l)})$,
but the converse is not necessarily true. In this paper we choose to find all the $\beta^l_{i,w}$ in order to find all the $P_w(Z)$ since it is not not significantly more difficult than to find the weaker, but sufficient, information obtained from all the $\phi_w(M^{(l)})$.}
\end{remark}

\section{Main theorem}

We are now able to give our main theorem, namely the higher weight spectra of the Veronese codes. We give here the result for $q \geqslant 4$, as well as the steps of the proof. Later, we will give the results for the degenerate cases $q=2,3$.

\begin{theorem}\label{mainqbig} Let $q\geqslant 4$ and consider the Veronese code $C_q$. Then all the $A_{w}^{(r)}$ are $0$, with the following exceptions:

\[\begin{array}{cc} 
A_{q^2-q}^{(1)}=\frac{q^4+2q^3+2q^2+q}{2} & A_{q^2}^{(1)}=q^5+q+1 \\
A_{q^2+q}^{(1)}=\frac{q^4-q}{2} & A_{q^2-1}^{(2)}=q^4+q^3+q^2 \\
A_{q^2}^{(2)}=q^3+2q^2+2q+1 &  A_{q^2+q-3}^{(2)}=\frac{q^8-q^6-q^5+q^3}{24} \\
A_{q^2+q-2}^{(2)}= \frac{q^7+q^6-q^4-q^3}{2} & A_{q^2+q-1}^{(2)}= \frac{q^8+5q^6+7q^5+4q^4-q^3-4q^2}{4}\\
A_{q^2+q}^{(2)}= \frac{2q^8+3q^7+q^6+4q^5+9q^4+5q^3-6q}{6} & A_{q^2+q+1}^{(2)}=\frac{3q^8+q^6-3q^5-q^3}{8} \\
A_{q^2}^{(3)}=q^2+q+1 &  A_{q^2+q-2}^{(3)}= \frac{q^6+2q^5+2q^4+q^3}{6} \\ 
A_{q^2+q-1}^{(3)}= \frac{q^7+2q^6+3q^5+3q^4+2q^3+q^2}{2} & A_{q^2+q}^{(3)}= \frac{2q^8+2q^7+3q^6+2q^5+4q^4+3q^3+2q^2}{2} \\
A_{q^2+q+1}^{(3)}=\frac{6q^9+3q^7+2q^6+q^5-5q^4+2q^3-3q^2}{6} & A_{q^2+q-1}^{(4)}=\frac{q^4+2q^3+2q^2+q}{2} \\
A_{q^2+q}^{(4)}= q^6+2q^5+2q^4+q^3+q^2+q+1 & A_{q^2+q+1}^{(4)}=\frac{2q^8+2q^7+2q^6+q^4-q}{2}\\
A_{q^2+q}^{(5)}=q^2+q+1 & A_{q^2+q+1}^{(5)}=q^5+q^4+q^3 \\
A_{q^2+q+1}^{(6)}=1
\end{array}\]
\end{theorem}

In order to prove this theorem, we will compute the Stanley-Reisner resolutions of the matroid $M_q$ and its elongations. We first will find which subsets of $\{1,\cdots,q^2+q+1\}$ that are minimal in the different $N_i$. In particular this will give us which Betti numbers $\beta^{(l)}_{1,j}$ are non-zero (Corollary \ref{cor:nonzerobetti}). When this is done, it  turns out that for every elongation $M_q^{(l)}$, for $l \ge 1$, the number of unknowns is equal to the number of Herzog-K\"uhl equations from Formula (\ref{HKBS}), and that all these equations are independent, For the matroid $M_q$ itself, however, there will be one unknown more than the number of equations. We will then, in Proposition \ref{result},  compute one of the missing Betti numbers $\beta^{(0)}_{2,q^2-1}$, After that we will be in a situation where  we can find all the Betti numbers with the Herzog-K\"uhl equations from Formula  (\ref{HKBS}). Thereafter we will compute the generalized weight polynomials $P(Z)$ using Theorem \ref{th:Jan}. Finally we will find the the higher weight spectra, using Theorem \ref{th:relspecpol} repeatedly.

\subsection{Stanley-Reisner resolutions}

We will use the following result by Hirschfeld~\cite{Hi} 
\begin{proposition}\label{pr:Hirschfeld}
In $\mathbb{P}^2_q$ the $\frac{q^6-1}{q-1}$ conics are as follows.  \begin{itemize} 
\item There are $q^2+q+1$ double lines,
\item There are $\frac{q(q+1)(q^2+q+1)}{2}$ pairs of two distinct lines
\item There are $q^5-q^2$ irreducible conics
\item There are $\frac{q(q-1)(q^2+q+1)}{2}$ conics that just possess a single $\Fq$-rational point each. 
\end{itemize}
\end{proposition}

There is a one-to-one correspondence between words of $C_q$ and affine equations for conics, and under this correspondence, the support of a codeword correspond to points of $\mathbb{P}_q^2$ that are not on the conic. Thus, the circuits of $M_q$ correspond to conics with maximal set of points (under inclusion). By Proposition~\ref{pr:Hirschfeld}, it is thus easy to see that we have two types of circuits, namely the one corresponding to pairs of lines, and the one corresponding to irreducible conics. This shows that \[\beta_{1,q^2+q+1-(2q+1)}^{(0)}= \frac{q(q+1)(q^2+q+1)}{2} \textrm{ and } \beta_{1,q^2+q+1-(q+1)}^{(0)} = q^5-q^2,\] the other $\beta_{1,j}^{(0)}$ being $0$. In order to compute the other Betti numbers of $M_q$, we will need the following lemma:

\begin{lemma}\label{lem:Huguesliker}
For any $X \subset E=\{1,\cdots,q^2+q+1\}$ the nullity $n(X)$ is equal to the dimension over $\Fq$ of the affine set of polynomial expressions that define conics that pass through all the points of $E\-X$.
\end{lemma}
\begin{proof}The matroid derived from any generator matrix of $C_q$, is the dual matroid of $M_q$. Its rank function $r^*$ therefore satisfies \[r(X) = |X| + r^*(E\-X) - r^*(E)\] for $X \subset E$, and hence $n(X) = r^*(E) - r^*(E\-X).$
The last expression is equal to the dimension of the kernel of the projection map
when projecting all the code words, each of which corresponds to the affine equation of a conic, on to the subspace of $\Fq^n$ indexed by $E\-X$. This kernel is precisely the polynomials that define conics passing through the points of $E\-X$, or alternatively, the codewords, whose support lie inside $X$.
\end{proof}

We can therefore find when the Betti numbers of $M_q$ and its elongations are non-zero. This comes as a corollary of the following theorem:

\begin{theorem}\label{th:minimalsubsets}
We have the following. \begin{itemize}
\item The minimal elements of $N_1$ are the complements of the $\frac{q(q+1)(q^2+q+1)}{2}$ pairs of distinct lines and of the $q^5-q^2$ irreducible conics. 
\item The  minimal subsets of $N_2$ are the $q^2(q^2+q+1)$ complements of $q+1$ points on a line and a point outside of the line, and the $\frac{(q^2+q+1)q^2(q^2+q)(q-1)^2}{24}$ complements of quadrilateral configurations of $4$ points such that no $3$ points lie on a line.\\
\item The minimal elements of $N_3$ are the $q^2+q+1$ complements of $q+1$ points on a line, and the $\frac{(q^2+q+1)q^2(q^2+q)}{6}$ complements of triangle configurations of $3$ non-aligned points.\\
\item The minimal elements of $N_4$ are the $\frac{(q^2+q+1)(q^2+q)}{2}$ complements of pairs of points.\\
\item The minimal elements of $N_5$ are the $q^2+q+1$ complements of a single point.\\
\item The only element of $N_6$ is $E$.
\end{itemize}
\end{theorem}

\begin{proof}
In the text following Proposition \ref{pr:Hirschfeld}, we have already treated the case with determining minimal elements of $N_1$. The complement of any set of points, such that no conic contains all of them, has nullity $0$ and is not considered here.\\
We will now determine the minimal elements of $N_2$. A subset of cardinality at least $q+3$ lying on a conic necessarily lies on a pair of lines, and defines  these two lines uniquely. Therefore, its complement has nullity $1$, and does not need to be considered here. Any subset of cardinality $q+2$ lying on a conic necessarily lies on a pair of distinct lines. If not $q+1$ of the points lie on the same line, then both lines are uniquely defined, and the nullity of the complement is $1$ again. If $q+1$ points lie on the same line, then there is an (exactly)  $2$-dimensional affine family of quadric polynomials which define conics going through these points (a fixed line and a variable line), and the nullity of the complement is $2$ by Lemma~\ref{lem:Huguesliker}. Obviously, the complement of these configurations are minimal in $N_2$. Moreover there are exactly $q^2(q^2+q+1)$ such configurations. Consider now $X \subset E$ with $5 \leqslant |X| \leqslant q+1$ that lie on a conic. If the points of $X$ lie on the same line, then $n(E\-X) = 3$ and it doesn't have to be considered here. If they lie on a pair of lines (but not a single line), then either $n(E\-X)=1$ if the two lines are uniquely defined, or $n(E\-X)=2$, but $E\-X$ is not minimal in $N_2$ (we could complete $X$ with the remaining points on the line that is uniquely defined). If they lie on an irreducible conic, then $n(E\-X)=1$ since an irreducible conic is uniquely defined by $5$ of its points. Consider now $X \subset E$ with $|X|=4$ and (then) lying on a conic. If $3$ of them are aligned, then we can argue in the same way as before for lines and pair of lines (so $E\-X$ is not minimal in any $N_i$). If no $3$ of them are aligned, then  there is a $2$-(and not $3$-)dimensional affine family of quadric polynomials defining  conics passing through $X$, and therefore $n(E\-X)=2$. Obviously, these configurations are minimal in $N_2$, since adding a point reduces the nullity (either being on a unique irreducible conic, or uniquely determined pairs of lines). There are exactly $\frac{(q^2+q+1)q^2(q^2+q)(q-1)^2}{24}$ such configurations. Finally, since the rank of the code is $6$, all subsets of cardinality at most $3$ have nullity at least $3$, and this completes the analysis of the minimal sets of $N_2.$ .\\

The other cases are done in a similar way.
Let us determine the minimal elements of $N_3$: The nullity of the complement of any subset of cardinality at least $q+2$ is at most $2$, as we have seen. The complement of $q+1$ points on a line, on the other hand,  are then minimal in $N_3$, and there are exactly $q^2+q+1$ lines in $\mathbb{P}_q^2$. The complements of any subset of cardinality between $q$ and $4$ has either nullity different from $3$ or are not minimal in $N_3$. Three non-aligned points give a $3$-dimensional affine family of quadric polynomials defining  conics passing through $X$, and the complement of the set of these points are minimal in $N_3$. There are $\frac{(q^2+q+1)q^2(q^2+q)}{6}$ such configurations. Finally, the complements of $2$ or less points have nullity at least $4$ since the rank of the code is $6$.\\
For nullity $4,5,6$, then we can see that $3$ points or more have complements with nullity at most $3$. And $i$ points give a $(6-i)$-dimensional affine family of quadric polynomials defining  conics passing through the $i$ points, for $i=2,1,0$ ,  Moreover there are  $\frac{(q^2+q+1)(q^2+q)}{2}$ pairs of points, $q^2+q+1$ single points and $1$ empty set in $\mathbb{P}_q^2,$
corresponding to $i=2,1,0,$ respectively. These observations settles the cases of finding the minimal elements of $N_4,N_5,N_6$.
\end{proof}
We recall that the length of the resolution of $R_{M_q}$ is $\dim C_q=6,$ and the lengths of the resolutions of $R_{M^{(i)}_q}$ then are $6-i$, for $i=1,\cdots,5.$
\begin{corollary}\label{cor:nonzerobetti}The only non-zero Betti numbers of $M_q^{(i)}$ for $0 \leqslant i \leqslant 5$ are $\beta_{0,0}^{(i)}=1$ and \[\beta_{1-i,q^2-q}^{(i)}, \beta_{1-i,q^2}^{(i)}, \beta_{2-i,q^2-1}^{(i)}, \beta_{2-i,q^2+q-3}^{(i)}, \beta_{3-i,q^2}^{(i)}, \beta_{3-i,q^2+q-2}^{(i)},  \beta_{4-i,q^2+q-1}^{(i)},  \beta_{5-i,q^2+q}^{(i)}\]
  $\beta_{6-i,q^2+q+1}^{(i)}$ when these quantities make sense. Moreover, we have \[\begin{array}{ll}
\beta_{1,q^2-q}^{(0)}= \frac{q(q+1)(q^2+q+1)}{2} & \beta_{1,q^2}^{(0)} = q^5-q^2\\
\beta_{1,q^2-1}^{(1)}= q^2(q^2+q+1) & \beta_{1,q^2+q-3}^{(1)} = \frac{(q^2+q+1)q^2(q^2+q)(q-1)^2}{24}\\
\beta_{1,q^2}^{(2)}=  q^2+q+1 & \beta_{1,q^2+q-2}^{(2)} = \frac{(q^2+q+1)q^2(q^2+q)}{6}\\
\beta_{1,q^2+q-1}^{(3)} = \frac{(q^2+q+1)(q^2+q)}{2} & \beta_{1,q^2+q}^{(4)} = q^2+q+1 \\ 
\beta_{1,q^2+q+1}^{(5)} = 1\end{array}\]
\end{corollary}

\begin{proof} This is an immediate consequence of Theorems~\ref{th:Euler},~\ref{thm:AAECC} and~\ref{th:minimalsubsets} and Proposition \ref{use}.
\end{proof}

As a corollary, we can find the generalized Hamming weights of the Veronese codes, already given in~\cite{Z}:

\begin{corollary} The generalized Hamming weights of the code $C_q$ are \[\begin{array}{cccccc} d_1= q^2-q, & d_2= q^2-1, & d_3=q^2, & d_4=q^2+q-1, & d_5=q^2+q, & d_6=q^2+q+1\end{array}.\]
\end{corollary}

\begin{proof}This is a direct consequence of Theorem~\ref{thm:AAECC}.
\end{proof}

After using Corollary \ref{cor:nonzerobetti} we have $7$ unknown remaining Betti number in the  $6$ (Herzog-K\"uhl) equations  described in Formula (\ref{HKBS}) for the matroid $M_q$, We have $5$ equations for $M_q^{(1)}$, with $5$ unknown Betti numbers, and for $2 \leqslant l \leqslant 5$, we have $6-l$ equations for $M_q^{(l)}$ for $5-l$ unknown Betti numbers. We will now find $\beta_{2,q^2-1}^{(0)}$, and thus reduce the number of unknown Betti numbers $\beta^{(0)}_{i,j}$ from $7$ to $6$. Thereafter, it turns out that all the Herzog-K\"uhl equation sets from Formula (\ref{HKBS}) will be independent, and we will find all the remaining unknown $\beta^{(l)}_{i,j}$, for $l=0,\cdots,5$ .

\begin{proposition}\label{prop:beta2}
Let $X \subset E$ be a set of $q+1$ points on a line together with a point outside of this line. Then \[\beta_{2,E\-X}^{(0)} = q.\]
\end{proposition}

\begin{proof} Write $X= D \cup \{P_0\}$ where $D$ is the line and $P_0$ the point outside. For ease of notation we denote $M_q$ by $M$.   We consider the restricted matroid $M_{E\-X}$ and will compute its Euler characteristic, and conclude by Theorem~\ref{thm:AAECC}. We will denote, for $0 \leqslant z \leqslant q^2-1$, \[D_z= \left|\left\{Y \subset E\-X:\ |\sigma|=z \textrm{ and } \sigma \not \in \I\right\}\right|.\] For $X \subset Z$ then $E\-Z \not \in \I$ if and only if $Z$ is contained in a conic, and necessarily this conic has to be a pair of lines containing $D$ and $P_0$. Thus, if $0 \leqslant z <q^2-q$, then $D_z = 0$. Also, $D_{q-1}=1$. Now, consider $q^2-q \leqslant z \leqslant q^2-2$. The pair of lines containing $X$ are parametrized by the points of $D$. And if $Z$ is a subset of such a parametrized conic of cardinality $t$, then we have $\binom{q-1}{t-(q+2)}$ choices for $Z$. Thus we find that \[D_z = (q+1)\binom{q-1}{q^2-1-z}.\] Using the fact that the alternate sums of binomial coefficients is $0$, we get that \[\chi(M_{E\-X}) = \sum_{z=0}^{q^2-1}(-1)^zD_z = (-1)^{q^2}q.\]
\end{proof}

\begin{corollary} \label{result}
We have \[\beta_{2,q^2-1}^{(0)} = q^3(q^2+q+1).\]
\end{corollary}
\begin{proof} This is a direct consequence of Theorem~\ref{th:minimalsubsets}: $\beta_{2,q^2-1}^{(0)}$ is the product of the number $q^2(q^2+q+1)$ of minimal elements of $N_2$ of degree $q^2-1$,  and the "local" contribution $\beta_{2,E\-X}=|\chi(M_{E\-X})| =|(-1)^{q^2}q|=q$ which we calculated in  Proposition \ref{prop:beta2}.
\end{proof}

\begin{theorem}\label{thm:betti}
With the previous notation, the Betti numbers of the matroid $M_q$ and its elongations are \[\begin{array}{cc} \beta_{1,q^2-q}^{(0)} = \frac{q^4+2q^3+2q^2+q}{2}, & \beta_{1,q^2}^{(0)} = q^5-q^2\\  \beta_{2,q^2-1}^{(0)} = q^5+q^4+q^3 ,& \beta_{2,q^2+q-3}^{(0)} = \frac{q^9-q^7-q^6+q^4}{24}\\
\beta_{3,q^2}^{(0)} = q^5-q^3-q^2+1 , & \beta_{3,q^2+q-2}^{(0)} = \frac{q^9-q^8-q^7+q^6+3q^5+3q^4}{6} \\  \beta_{4,q^2+q-1}^{(0)} = \frac{q^9-2q^8+q^7+3q^6+2q^5-q^4-4q^3}{4} ,  & \beta_{5,q^2+q}^{(0)} = \frac{q^9-3q^8+5q^7-q^6-3q^5-2q^4+6q^2-3q}{6} \\
\beta_{6,q^2+q+1}^{(0)} = \frac{q^9-4q^8+11q^7-17q^6+12q^5-3q^4}{24}
\end{array}\]
\[\begin{array}{cc}\beta_{1,q^2-1}^{(1)} = q^4+q^3+q^2,& \beta_{1,q^2+q-3}^{(1)} = \frac{q^8-q^6-q^5+q^3}{24} \\
\beta_{2,q^2}^{(1)} = q^4+q^3-q-1 ,& \beta_{2,q^2+q-2}^{(1)} = \frac{q^8+q^6+3q^5+4q^4+3q^3}{6} \\
\beta_{3,q^2+q-1}^{(1)} = \frac{q^8+3q^6+3q^5-3q^3-4q^2}{4},& \beta_{4,q^2+q}^{(1)} = \frac{q^8+5q^6-q^5-6q^4-5q^3+6q}{6} \\
\beta_{5,q^2+q+1}^{(1)} = \frac{q^8+7q^6-9q^5-8q^4+9q^3}{24} 
\end{array}\]
\[\begin{array}{cc}
\beta_{1,q^2}^{(2)} = q^2+q+1 ,& \beta_{1,q^2+q-2}^{(2)} = \frac{q^6+2q^5+2q^4+q^3}{6} \\
\beta_{2,q^2+q-1}^{(2)} = \frac{q^6+2q^5+2q^4+q^3}{2},& \beta_{3,q^2+q}^{(2)} = \frac{q^6+2q^5+2q^4-q^3-2q^2-2q}{2} \\ 
\beta_{4,q^2+q+1}^{(2)} = \frac{q^6+2q^5+2q^4-5q^3}{6}
\end{array}\]
\[\begin{array}{ccc}
\beta_{1,q^2+q-1}^{(3)} = \frac{q^4+2q^3+2q^2+q}{2},& \beta_{2,q^2+q}^{(3)} = q^4+2q^3+q^2-1,& \beta_{3,q^2+q+1}^{(3)} = \frac{q^4+2q^3-q}{2}
\end{array}\]
\[\begin{array}{ccc} \beta_{1,q^2+q}^{(4)} = q^2+q+1 ,& \beta_{2,q^2+q+1}^{(4)} = q^2+q, & \beta_{1,q^2+q+1}^{(5)} = 1
\end{array}\]
\end{theorem}

\begin{proof}
This follows immediately from Corollary~\ref{cor:nonzerobetti}, Proposition~\ref{prop:beta2} and Theorem~\ref{th:HK}, after using the computer program Mathematica to solve the Herzog-K\"uhl equations (\ref{HKBS}) from Theorem \ref{th:HK} for the Betti numbers appearing in each of the the $\mathbb{N}$-graded resolutions of the Stanley-Reisner rings of the matroids $M_q^{(l)},$ for $l=0,1,\cdots,5$. (After usage of Corollary~\ref{cor:nonzerobetti} which assigns integer values to a sufficient set of  Betti numbers, the coefficient matrices of the Herzog-K\"uhl equations for each of the matroids in question, in terms of those Betti numbers that are still unknown, are now of vandermonde type).
\end{proof}
\begin{remark}
{\rm It is also possible to find all these Betti numbers without using the Herzog-K\"uhl equations:
First Proposition \ref{use} gives, for each $l$ and $i$ in question,  that a subset $Y$ of $E$ is minimal among those sets that have nullity $i$ for the matroid  $M_q^{(l)}$ if and only if $Y$  is minimal among those sets that have nullity $i+l$ for the matroid  $M_q^{(l)}.$
Furthermore one can find the local contributions  $\beta^{(l)}_{i,Y}$,
 for each $Y$ minimal among those sets that have nullity $i$ for the matroid $M_q^{(l)},$ 
 by performing arguments and calculations analogous to those in the proof of Proposition \ref{prop:beta2}.
 The result, $\beta^{(l)}_{i,j}$, is then computed as the product of the number (given in Theorem \ref{th:minimalsubsets})  of subsets $Y$ of $E$ that have cardinality $j$ and are minimal in $N_{i+l},$
 and the common number $\beta^{(l)}_{i,Y}$ for all these sets $Y$. We have done this for all the Betti numbers given in Theorem \ref{thm:betti}, but see no reason to present the calculations here, since usage of a computer program like Mathematica gives the solution for the  $\beta^{(l)}_{i,j}$ directly. If, on the other hand, for some reason, one would be interested in knowing the values of the "local" contributions $\beta^{(l)}_{i,Y}$, one can just divide the values of the $\beta^{(l)}_{i,j}$
  appearing in Theorem \ref{thm:betti} by the corresponding numbers appearing in Theorem \ref{th:minimalsubsets}. }

\end{remark}

\eject

\subsection{Higher weight polynomials and weight spectra}

\begin{theorem}\label{th:polynomials}
Let $q \geqslant 4$ be a prime power. Then  the Veronese code $C_q$ has 9 non-zero generalized weight polynomials, namely
\begin{eqnarray*}
P_0(Z)&=&1 \\
P_{q^2-q}(Z) &=& \binom{q^2+q+1}{2}(Z-1) \\
P_{q^2-1}(Z) &=& (q^2+q+1)q^2(Z-q)(Z-1) \\
P_{q^2}(Z)&=& (q^2+q+1)(Z-1)(Z^2-(q^2-1)Z + 2q^3-2q^2-q+1)\\
P_{q^2+q-3}(Z)&=& \frac{(q^2+q+1)(q+1)q^3(q-1)^2(Z-q)(Z-1)}{24}\\
P_{q^2+q-2}(Z)&=& \frac{(q^2+q+1)(q+1)q^3(Z-1)(Z-q)(Z-(q^2-3q+3))}{6}\\
P_{q^2+q-1}(Z) &=& \frac{(q^2+q+1)(q+1)q(Z-1)(Z-q)(2Z^2-2(q^2-q)Z+(q^4-4q^3+7q^2-4q))}{4}\\
\frac{6P_{q^2+q}(Z)}{(q^2+q+1)(Z-1)} &=&6Z^4- (6q^2+6q-6)Z^3+(3q^4+3q^3-6q)Z^2\\&&-(q^6-q^5+5q^4-5q^3-6q^2+6q)Z\\
&&+(q^7-4q^6+8q^5-5q^4-6q^3+9q^2-3q)\\
\frac{24P_{q^2+q+1}}{(Z-1)(Z-q)} &=& 24Z^4-24q^2Z^3+(12q^4-12q)Z^2\\
&&-(4q^6-4q^5+8q^4-20q^3+12q^2)Z\\&& + (q^8-4q^7+11q^6-17q^5+12q^4-3q^3)
\end{eqnarray*} 
\end{theorem} 

\begin{proof}
This is a direct consequence of Theorems~\ref{th:Jan} and~\ref{thm:betti}.
\end{proof}

\begin{proof}[Proof of Theorem~\ref{mainqbig}]
This is a direct consequence of Theorem~\ref{th:polynomials} and repeated usage of  Theorem~\ref{th:relspecpol}.
\end{proof}

\section{The cases of binary and trinary codes} \label{2and3}

The cases $q=2$ and $q=3$ are very similar to  the "general" case $q \ge 4$, except that some degeneracies appear. It can be shown that in the case $q=2$, where  $q^2-1=q^2+q-3$ and $q^2=q^2+q-2$, we have $\beta_{1,4}^{(0)} = 0$, and all the resolutions in question are linear, and easy to cope with (The code is MDS for $q=2$, and then both $M$ and all its elongation matroids are uniform, and their associated Betti numbers then follow directly from  the Herzog-K\"uhl equations). 

In the case $q=3$, we have $\beta_{1,9}^{(0)}=0$. This constitutes a difference with the cases $q \ge 4$, where the coefficient  $\beta_{1,q^2}^{(0)}$ is non-zero.  The non-zero value is due to the complement $X$ of the irreducible conic (with $q+1$ points). For $q \geqslant 4$, these complements are minimal sets in $N_1$. But for $q=3$ an irreducible conic has $4$ points, and is always included in a pair of distinct lines, and therefore would not lead to a minimal element in $N_1$.
Apart from this difference from the cases $q \ge 4$ the arguments for establishing the Betti numbers, generalized weight polynomials, and higher weight spectra are almost identical for $q=3$ to those in the cases $q \ge 4$. 

We now give just the main result about these 2 cases, without going more into the details concerning the computation of the Betti numbers and the general weight polynomials:

\begin{theorem}\label{mainq3}
The higher weight spectra of the Veronese code $C_3$ is
\[\begin{array}{cccc}
A_6^{(1)}=78, & A_9^{(1)}=247, & A_{12}^{(1)}=39,&A_8^{(2)}=117\\
A_9^{(2)}=286,& A_{10}^{(2)}=1404,&  A_{11}^{(2)}=3042,& A_{12}^{(2)}=3705\\
A_{13}^{(2)}=2457,& A_9^{(3)}=13,& A_{10}^{(3)}=234,& A_{11}^{(3)}=2340\\
A_{12}^{(3)}=10296,& A_{13}^{(3)}=20997,& A_{11}^{(4)}=78,& A_{12}^{(4)}=1417\\
A_{13}^{(4)}=9516,& A_{12}^{(5)}=13,&  A_{13}^{(5)}=351,&A_{13}^{(6)}=1,
\end{array}\] all the other being $0$.
\end{theorem}

\begin{theorem}\label{mainq2}
The higher weight spectra of the Veronese code $C_2$ is
\[\begin{array}{cccc}
A_{2}^{(1)}=21, & A_{4}^{(1)}= 35 ,& A_{6}^{(1)}= 7,& A_{3}^{(2)}=35\\
A_{4}^{(2)}=105,& A_{4}^{(3)}=35,& A_{5}^{(2)}=210,& A_{5}^{(3)}=210\\ A_{5}^{(4)}=21,& A_{6}^{(2)}=210,& A_{6}^{(3)}=560,& A_{6}^{(4)}=175\\
A_{6}^{(5)}=7,& A_{7}^{(2)}=91,& A_{7}^{(3)}=590,& A_{7}^{(4)}=455\\  &A_{7}^{(5)}=56,& A_{7}^{(6)}=1 
\end{array}\] all the other being $0$.
\end{theorem}

\end{document}